\documentclass[10pt,twoside]{article}
\usepackage{amsfonts}
\usepackage{fancyhdr}
\usepackage{titlesec}
\usepackage{cite}
\usepackage{ifthen}
\usepackage{pifont}
\usepackage{stmaryrd}
\usepackage{setspace}
\usepackage{indentfirst}
\usepackage{amsmath,amssymb,amscd,bbm,amsthm,mathrsfs,dsfont}
\input amssym.def

\titleformat{\section}{\centering\large\bfseries}{\S\arabic{section}}{1em}{}
\newboolean{first}
\setboolean{first}{true}

\newtheorem{theorem}{Theorem}[section]

\newtheorem{definition}[theorem]{Definition}

\newtheorem{lemma}[theorem]{Lemma}

\renewenvironment{proof}[1]{\noindent\textbf{#1}}{\ \rule{0.5em}{0.5em}}

\begin{document}
\title{\bf \Large Infinite collisions of simple random walks on random recursive trees generated by Bernoulli sequences \author{Jia-shan Tang$^{*}$,\  Feng Wang,\  Xian-Yuan Wu}\date{}} \maketitle
\footnote{*Corresponding author}
 \footnote{MR Subject Classification(2010):60J10, 05C81 .}
 \footnote{Supported by the National Natural Science Foundation of China\, (No. 11471222, 61973015) and by the Beijing Outstanding Young Scientist Program (No. JWZQ20240101027).}
\begin{center}
\begin{minipage}{135mm}

{\bf \small Abstract}.\hskip 2mm {\small
In this paper, we study random recursive trees generated by Bernoulli sequences. Starting from a graph with two vertices and one edge, each new vertex is connected to the last vertex with probability $ p $, or to the second-last vertex with probability $ q = 1-p $, this recursive construction yields a random infinite recursive tree $T$. We prove that $T$ almost surely has exactly one topological end. Furthermore, we establish that $T$ has the infinite collision property:  two independent simple random walks on $T$ collide infinitely often almost surely.\vbox{}\\}
\end{minipage}
\end{center}

\section{Introduction and the main result}
Let $G$ be an infinite connected graph, and let $ X = \{X_n\} $ and $ X' = \{X'_n\} $ be two independent simple random walks on $G$, we say that $G$ has the \textit{infinite collision property} if $ | \{ n : X_n = X'_n \} | = \infty $ (i.e. occupy the same vertex at the same time) almost surely. Otherwise, we say $G$ has the finite collision property. The infinite collision property generalizes the classical notion of recurrence. Its study reveals deeper structural characteristics of stochastic processes.

Although transitive recurrent graphs such as $\mathbb{Z}$ and $\mathbb{Z}^2$ are easily seen to have the infinite collision property, Krishnapur and Peres\cite{KP2004} first proved that two independent simple random walks on $Comb(\mathbb{Z})$ have the finite collision property. This shows that the infinite collision property is non-monotonic. Barlow, Peres, and Sousi\cite{BPS2010} gave a sufficient condition for the infinite collision property. Chen and Chen\cite{CC2010} proved the infinite collision property on the open cluster of $\mathbb{Z}^2$ and\cite{CC2011} further extended the previous results. Hutchcroft and Peres\cite{HP2015} prove that in any recurrent reversible random rooted graph, two independent simple random walks started at the same vertex collide infinitely often almost surely. Richey\cite{R2018} discusses collisions between Brownian motion particles.

Halberstam and Hutchcroft\cite{HH2022} studied the dynamic random conductance model on $\mathbb{Z}^2$. While all prior work assumed a static environment , they pioneered the study of collision problems in dynamically random environments. Watanabe\cite{W2023} established quantitative estimates for the collision times of two independent simple random walks on the Uniform Spanning Tree (UST) over $\mathbb{Z}^3$, which yields a novel proof of the infinite collision property. This approach is distinct from existing methods based on Green’s functions. By exploiting the duality between the voter model and coalescing random walks, Astoquillca\cite{A2024} connected collision properties to the stationary measures of interacting particle systems. De Ambroggio, Nitzschner and Scali\cite{DNS2025} characterizes phase transitions of the property on several types of comb graphs with planar bases. The recent paper by Croydon, Shiraishi and Watanabe\cite{CSW2026} extended the scope of the infinite collision property to random walk traces, investigating collision behaviors of random walks on the random graph formed by the trace of four-dimensional simple random walks. In this paper, we focus on the infinite collision property on random recursive trees $T$.

The random recursive trees $T$ is constructed recursively as follows:

(1) Start with the initial graph $T_1=(V_1, E_1)$, where $V_1 = \{v_0, v_1\}$ and $E_1 = \{\langle v_0, v_1 \rangle\}$.

(2) Let $T_{n}=(V_{n},E_{n})$, at step $ n$, add a new vertex $ v_{n+1} $, so that $ V_{n+1} = V_n \cup \{v_{n+1}\} $. The new vertex $ v_{n+1}$ connects to $ v_n $ with probability $ p $, or to $ v_{n-1} $ with probability $ q = 1-p$.

Formally, $E_{n+1} = E_n \cup \{\langle v_{n+1}, w_n \rangle\} $, where the independent random variable $ w_n $ satisfies:
$$
    \mathbb{P}(w_n = v_n) = p, \quad \mathbb{P}(w_n = v_{n-1}) = q = 1-p.
$$
Let $T = \bigcup_{n=1}^{\infty} T_n$. The random variable $\{w_n\}$ follows a Bernoulli distribution.

To characterize the spatial geometry of such random infinite trees, we introduce the notion of topological ends. A topological end intuitively describes a "boundary point at infinity" of an infinite graph. This concept was originally introduced by H. Freudenthal \cite{F1942}, and its precise definition is stated in Definition \ref{def1}. Pemantle\cite{P2004} investigates the topological ends of spanning trees in $\mathbb{Z}^d$, when $d \leq 4$, the tree has only one topological end, and when $d \geq 5$ each component of the spanning forest having one or two topological ends. Our first main result concerns the number of topological ends of $T$.

\begin{theorem}\label{thm1}
 $ \mathbb{P}(\text{The random recursive tree  } T \text{ has exactly one topological end}) = 1 $.
\end{theorem}

We observe that T shares a similar structure with comb graphs, and try to generalizes the method developed in \cite{CC2011} for infinite collisions of two random walks on random recursive trees. Now we state the main result of the paper as the following:
\begin{theorem}\label{thm2}
The random recursive tree $ T $ has the infinite collision property.
\end{theorem}

Research on the infinite collision property is also valuable in other fields. In physics and statistical mechanics, this property determines whether particles can interact infinitely often, which directly governs the phase transition behavior of systems. In network science and graph theory, the meeting efficiency of two random walk agents affects the performance of algorithms such as gossip protocols, random search and web crawlers. In mathematical biology and ecology, this property is, in a sense, related to whether a population can maintain genetic diversity over an infinite time horizon.

The rest of the paper is organized as follows. In Section 2, we introduce topological ends and prove Theorem \ref{thm1}. In Section 3, we first analyze the properties of the trunk and branches of $T$, and then discuss the applicability of the sufficient condition given by Chen and Chen \cite{CC2011} to $T$. Finally, we present the proof of Theorem \ref{thm2}.

\section{Proof of Theorem 1.1}
We first introduce some notation that will be used throughout the paper. For $G=(V, E)$, a sequence of sides is called a path $P(v_{1},v_{n})$ from $v_{1}$ to $v_{n}$ in $G$ if $\{\langle v_{i}, v_{i+1}\rangle\}_{1\leq i\leq n-1}\in E$. The length of the path $L(P(v_{1},v_{n}))$ is defined as the number of sides of $P(v_{1},v_{n})$. $P(v_{1},\infty)$ denote a path from $v_1$ to infinity, which is abbreviated as $P_{v_1}$. A path is said to be a simple path if it contains no repeated vertices. The distance between two vertex sets $u,v\in G$ is defined by $d(u,v):=\min\{L(P(u,v))\}$. Similarly, we define the distance from a point to a set and that between two sets.

Geometrically, an end of a suitable topological space is a point at infinity. The Definition 2.1 of \cite{AQR2026} is described as follows:
\begin{definition}\label{def1}
Let $X$ be a locally compact, locally connected, connected, and Hausdorff space, and let $\{U_n\}_{n\in \mathbb{N}}$ be an infinite nested sequence $U_1 \supset U_2 \supset \cdots$ of non-empty connected open subsets of $X$, such that the following hold:

(1) For each $n \in \mathbb{N}$, the boundary $\partial U_n$ of $U_n$ is compact,

(2) The intersection $\cap_{n\in \mathbb{N}} U_n = \emptyset$,

(3) For each compact $K \subset X$ there is $m \in \mathbb{N}$ such that $K \cap U_m = \emptyset$.
\end{definition}

Two nested sequences $\{U_n\}_{n\in \mathbb{N}}$ and $\{U'_n\}_{n\in \mathbb{N}}$ are equivalent if for each $n \in \mathbb{N}$ there exist $j, k \in \mathbb{N}$ such that $U_k \subset U'_n$, and $U'_j \subset U_n$. The corresponding equivalence classes $ [U_n]_{n\in \mathbb{N}}$ of these sequences are called the topological ends of $X$.

According to the above definition, the definition of the number of topological ends of a tree is given in \cite{P2004} by R. Pemantle:

\begin{definition}\label{def2}
For a tree, the number of topological ends is defined as the number of infinite simple paths from any fixed vertex.
\end{definition}

Since $T$ is connected, this number is the same for every vertex, i.e., the number of topological ends is translation-invariant on $T$.  The sequence of choices $ \{w_n\} $ can be viewed as a sequence of Bernoulli trials $ \{b_n\} $, where
$$
  \mathbb{P}(b_n = 1) = p, \quad \mathbb{P}(b_n = 0) = q = 1-p.
$$
and the event $ \{b_n = 1\}$ corresponds to $\{w_n = v_n\}$, while $\{b_n = -1\}$ corresponds to $\{w_n = v_{n-1}\}$. Thus, we can generate a corresponding random infinite tree $ T $ based on a random infinite Bernoulli sequence $ \{b_n\} $. We now present the proof of Theorem \ref{thm1}.

\begin{proof}
P\textbf{roof: }\label{pofthm1}We first observe that once vertex $ v_{n+2} $ is added to $ T_{n+1} $ to form $ T_{n+2} $, the degree of every vertex in $ V_n $ becomes fixed, and no further descendants can be generated from them. In other words, each vertex has at most two opportunities to produce new descendants. Hence, $T$ has at most two topological ends.

We suppose that $T$ has two topological ends. 

Let $ P_{v_0} $ and $ P'_{v_0} $ be two such paths from $ v_0 $. Define $ N_2 := \max\{i : v_i \in P_{v_0} \cap P'_{v_0}\} $. Then there exist $ n_1, n_2 > N_2 $ such that $ v_{n_1} \in P_{v_0} $ and $ v_{n_2} \in P'_{v_0} $, both adjacent to $ v_{N_2} $. By the construction rules, the only possible later neighbors of $ v_{N_2} $ are $ v_{N_2+1} $ and $ v_{N_2+2} $. Without loss of generality, let $ v_{N_2+1} \in P_{v_0} $, which forces $ v_{N_2+2} \in P'_{v_0} $. If $ v_{N_2+3} \in P'_{v_0} $, then $ v_{N_2+1} $ would have no descendants, contradicting $ v_{N_2+1} \in P_{v_0} $. Thus, $ v_{N_2+3} \in P_{v_0} $. Continuing this argument inductively shows that for all $ n\ge 1, v_{N_2+2n-1}\in P_{v_0},v_{N_2+2n}\in P'_{v_0}$, i.e., for all $ n \ge N_2+1 $, we must have $ b_n = 0$. Therefore,
$$
\mathbb{P}(T \text{ has two topological ends}) \le \mathbb{P}(b_n = 0 \text{ for all } n \ge N_2+1) = \lim_{n \to \infty} q^{n-N_2-1} = 0.
$$

This completes the proof that $T$ has exactly one topological end almost surely.
\end{proof}

\section{Proof of Theorem 1.2}\label{sec3}

In the previous section, we proved Theorem 1.1, which states that $T$ has exactly one simple path from $v_0$ to infinity almost surely. This result characterizes the geometric structure of $T$ as $n$ tends to infinity. Naturally, we proceed to further elaborate on the structure of $T$ in this section. We first discuss the detailed relationship between the trunk and branches of $T$ with the random Bernoulli sequence $\{b_n\}$. Based on this analysis, we establish several lemmas required for subsequent proofs, and finally present two different proofs for Theorem \ref{thm2}.

\subsection{The trunk and branches of $T$}

From the Theorem \ref{thm1}, $T$ has exactly one simple path from $ v_0 $ to infinity with probability one. This implies that $ P_{v_0} $ is unique almost surely, and we refer to it as the \textbf{trunk} of $T$. We now characterize the conditions under which a vertex $ v_n$ lies on the trunk $P_{v_0} $ in terms of the sequence $ \{b_n\} $.

For any $ n > 0 $, if $ b_n = 1 $, then the edge $ \langle v_n, v_{n+1} \rangle $ belongs to $T$. Consequently, $ v_{n+2} $ must be a descendant of $ v_n $. By construction, all vertices added thereafter are descendants of either $ v_{n+1} $ or $ v_{n+2} $, so every vertex in $ \{v_i\}_{i \geq n+1} $ is a descendant of $ v_n $. This implies that $ v_n $ lies on the trunk $P_{v_0} $.

Next, we analyze the case where $ b_n = 0 $ and determine the condition for $ v_n$ to lie on the trunk $P_{v_0} $. 

Define $ n_1 = \max\{i : b_i = 1, i < n\} $ and $ n_2 = \min\{i : b_i = 1, i > n\} $. From the previous discussion, we know $ v_{n_1}, v_{n_2} \in P_{v_0} $. If $n_2 - n_1 $ is odd, then $ v_{n_1}, v_{n_1+1}, v_{n_1+3}, \ldots, v_{n_1+2k-1} = v_{n_2} $ all belong to $ P_{v_0} $. Consequently, $ v_n \in P_{v_0} $ if $ n - n_1 $ is also odd, and $ v_n \notin P_{v_0} $ if $ n - n_1 $ is even. Similarly, if $ n_2 - n_1 $ is even, then $ v_{n_1}, v_{n_1+2}, v_{n_1+4}, \ldots, v_{n_1+2k} = v_{n_2} $ all lie on $ P_{v_0} $. Consequently, $ v_n \in P_{v_0} $ if $ n - n_1 $ is even, and $ v_n \notin P_{v_0} $ if $ n - n_1 $ is odd.

In summary, $ v_n \notin P_{v_0} $ if and only if $ n_2-n_1 $ and $ n-n_1 $ have the \textbf{same} parity. Equivalently, this condition is equivalent to $n_2 - n $ being even.

Thus, we conclude:
$$
\begin{cases}
v_n \in P_{v_0}, & \text{if } b_n = 1, \\
v_n \in P_{v_0}, & \text{if } b_n = 0 \text{ and } n_2 - n \text{ is odd},\\
v_n \notin P_{v_0}, & \text{if } b_n = 0 \text{ and } n_2 - n \text{ is even}.
\end{cases}
$$

Therefore, given any random sequence $ \{b_n\}$, we can determine whether a vertex $ v_n $ lies on the trunk $ P_{v_0} $ of the random infinite tree $T$ generated by $ \{b_n\} $, either by checking the value of $ b_n $ or examining the parity of the distance from $ n $ to the next index at which the sequence equals 1.

For the trunk $P_{v_0}$ of $T$, there exists a strictly increasing sequence of non-negative integers $\{\sigma_i\}_{i \geq 0}$ such that $P_{v_0}(i) = v_{\sigma_i}$, with $\sigma_0 = 0$. We now analyze the transition probabilities for the sequence $\{\sigma_i\}$.

Observe that, in $T$, each vertex $v_{\sigma_i}$ has at most two descendant neighbors: $v_{\sigma_i + 1}$ and $v_{\sigma_i + 2}$. Consequently, $\sigma_{i+1} - \sigma_i \leq 2$.

If $b_{\sigma_i} = 0$, then $\langle v_{\sigma_i - 1}, v_{\sigma_i + 1} \rangle \in T$. Since $v_{\sigma_i} \in P_{v_0}$, we also have $\langle v_{\sigma_i}, v_{\sigma_i + 2} \rangle \in T$, so $\sigma_{i+1} = \sigma_i + 2$.

If $b_{\sigma_i} = 1$, then the edge $\langle v_{\sigma_i}, v_{\sigma_i + 1} \rangle \in T$.

$\bullet$ If additionally $b_{\sigma_i + 1} = 1$, then $\langle v_{\sigma_i + 1}, v_{\sigma_i + 2} \rangle \in T$, hence $\sigma_{i+1} = \sigma_i + 1$.

$\bullet$ If $b_{\sigma_i + 1} = 0$:

$\quad \circ$ When $b_{\sigma_i + 2} = 1$, $v_{\sigma_i + 1}$ has no descendants, so $\sigma_{i+1} = \sigma_i + 2$.

$\quad \circ$ When $b_{\sigma_i + 2} = 0$, and $b_{\sigma_i + 3} = 1$, $v_{\sigma_i + 2}$ has no descendants, so $\sigma_{i+1} = \sigma_i + 1$.

By induction, if $b_{\sigma_i} = 1$, then $\sigma_{i+1} = \sigma_i + 2$ if the subsequent run of $-1$s has odd length. Otherwise, $\sigma_{i+1} = \sigma_i + 1$. Thus,
$$
P(\sigma_{i+1} = \sigma_i + 1)= p^2 \sum_{k=0}^{\infty} q^{2k}= \frac{p}{1 + q}.
$$

$$
P(\sigma_{i+1} = \sigma_i + 2)= q + p^2 q \sum_{k=0}^{\infty} q^{2k}= q + \frac{pq}{1 + q} = \frac{2q}{1+q}.
$$
 
In summary:

\begin{eqnarray}
\begin{cases}
P(\sigma_{i+1} = j + 1 \mid \sigma_i = j) = \dfrac{p}{1+q},\\
P(\sigma_{i+1} = j + 2 \mid \sigma_i = j) = \dfrac{2q}{1+q}.
\end{cases}
\end{eqnarray}

If there exists a finite simple path $P'(v)\subset T$ such that $P'(v)\cap P_{v_0} = \{v\}$, we call $P'(v)$ a \textbf{branch} of $v$. For each vertex $v_{\sigma_i}$, let $L_i$ denote the length of its associated branch.

From the proof \ref{pofthm1} of Theorem \ref{thm1}, $T$ has two topological if and only if $b_n = 0$ for all sufficiently large $n$. Consequently, finite branch arise precisely from finite runs of $0$s in $\{b_n\}$. We now characterize this relationship.

Without loss of generality, suppose $b_{\sigma_i} = 1$ initiates a run of $0$s of length $k$. Then $b_{\sigma_i+j} = 0, 1 \leq j \leq k, b_{\sigma_i+k+1} = 1$.\\
$\bullet$ For $k = 1$: $b_{\sigma_i} = 1, b_{\sigma_i + 1} = 0, b_{\sigma_i + 2} = 1$. All subsequent vertices descend from $v_{\sigma_i + 2}$ or $v_{\sigma_i + 3}$, while $v_{\sigma_i + 1}$ generates no further descendants. Thus, the corresponding branch is $\{v_{\sigma_i}, v_{\sigma_i + 1}\}$, so $L_i = 1$.\\
$\bullet$ For $k = 2$: $b_{\sigma_i} = 1,  b_{\sigma_i + 1} = 0, b_{\sigma_i + 2} = 0, b_{\sigma_i + 3} = 1$. The branch is $\{v_{\sigma_i}, v_{\sigma_i + 2}\}$, so $L_i = 1$.\\
$\bullet$ For $k = 3$: $b_{\sigma_i} = 1,  b_{\sigma_i + 1} = 0, b_{\sigma_i + 2} = 0, b_{\sigma_i + 3} = 0, b_{\sigma_i + 4} = 1$. The branch is $\{v_{\sigma_i}, v_{\sigma_i + 1},\\ v_{\sigma_i + 3}\}$, so $L_i = 2$.\\
$\bullet$ For $k = 4$: $b_{\sigma_i} = 1, b_{\sigma_i + 1} = 0, \ldots, b_{\sigma_i + 4} = 0, b_{\sigma_i + 5} = 1$. The branch is $\{v_{\sigma_i}, v_{\sigma_i + 2}, v_{\sigma_i + 4}\}$, so $L_i = 2$.

By induction, if $k = 2m - 1$, the branch is $\{v_{\sigma_i}, v_{\sigma_i + 1}, v_{\sigma_i + 3}, \ldots, v_{\sigma_i + 2m - 1}\}$, and $L_i = m$. If $k = 2m$, the branch is $\{v_{\sigma_i}, v_{\sigma_i + 2}, v_{\sigma_i + 4}, \ldots, v_{\sigma_i + 2m}\}$, and $L_i = m$.

Thus, $L_i = m$ if and only if $b_{\sigma_i}$ is followed by a run of $0$s of length $2m-1$ or $2m$. Hence,

$$
\begin{aligned}
P(L_i = m) &= P(b_{\sigma_i}=1, b_{\sigma_i+1}=0, \ldots, b_{\sigma_i+2m-1}=0, b_{\sigma_i+2m}=1) \\
&\quad + P(b_{\sigma_i}=1, b_{\sigma_i+1}=0, \ldots, b_{\sigma_i+2m}=0, b_{\sigma_i+2m+1}=1) \\
&= p^2 q^{2m-1} + p^2 q^{2m} = p^2 (1 + q) q^{2m-1}.
\end{aligned}
$$

The probability that $L_i = 0$ is:

$$
\begin{aligned}
P(L_i = 0) &= 1 - \sum_{m=1}^{\infty} P(L_i = m) = 1 - \sum_{m=1}^{\infty} p^2 (1 + q) q^{2m-1} \\
&= 1 - p^2 (1 + q) \frac{q}{1 - q^2} = 1 - pq.
\end{aligned}
$$

Since $p + q = 1$, we have $P(L_i = 0) = p + q - pq = q + p(1 - q) = q + p^2$.

There are two scenarios for $L_i = 0$:

$\bullet$ If $b_{\sigma_i} = 0$, then the in-degree of $v_{\sigma_i}$ is at most 1, so $v_{\sigma_i}$ is either on $P_{v_0}$ or part of a branch of a previous vertex, implying $L_i = 0$.

$\bullet$ If $b_{\sigma_i} = 1$ and $b_{\sigma_i+1} = 1$, then $v_{\sigma_i} \in P_{v_0}$ and $L_i = 0$.

Thus, $P(L_i = 0) = P(b_{\sigma_i} = 0) + P(b_{\sigma_i} = 1, b_{\sigma_i+1} = 1) = q + p^2$, which matches the calculation above.

In conclusion, the distribution of the branch length $L_i$ for $v_{\sigma_i}$ is:
 
\begin{eqnarray}
P(L_i = m) =
\begin{cases}
q + p^2, & m = 0, \\
p^2 (1 + q) q^{2m-1}, & m \neq 0.
\end{cases}\label{001}
\end{eqnarray}

The expected length of the branch is:

$$
\begin{aligned}
E(L_i) &= \sum_{m=0}^{\infty} m P(L_i = m) = \sum_{m=1}^{\infty} m p^2 (1 + q) q^{2m-1}\\
&= p^2 q (1 + q) \sum_{m=1}^{\infty} m (q^2)^{m-1}\\
&= \frac{p^2 q (1 + q)}{(1 - q)^2 (1 + q)^2}= \frac{q}{1 + q}.
\end{aligned}
$$

The above expressions are independent of $\sigma_i$, indicating that the branch length for any vertex in $T$ (except $v_0$) follows the same distribution. 

For the initial graph $G_1$, we have $b_0 = 1$ almost surely. Therefore, the distribution of the branch length $L_0$ for $v_0$ is:

\begin{eqnarray}
P(L_0 = m) =
\begin{cases}
p, & m = 0,\\
p (1 + q) q^{2m-1}, & m \neq 0.
\end{cases}
\end{eqnarray}Its expectation is $E(L_0) = \dfrac{q}{p(1+q)} = \dfrac{q}{1-q^2}$.

Based on the above discussion, we readily obtain the following lemma.
\begin{lemma}\label{lem1}
In $T$, if $v_{\sigma_i}$ and $v_{\sigma_j}$ are the root vertices of two adjacent branches with $i < j$, then $L_i \leq j - i$.
\end{lemma}

\begin{proof}
P\textbf{roof: }Note that if both $v_{\sigma_i}$ and $v_{\sigma_j}$ have branches, then necessarily $b_{\sigma_i} = 1$ and $b_{\sigma_j} = 1$. As shown earlier, $L_i = m$ if and only if there is a run of $0$s of length $2m-1$ or $2m$ after $b_{\sigma_i}$. Thus, the branch length $L_i$ equals the ceiling of half the length of that run, so $L_i \leq \lceil (\sigma_j - \sigma_i)/2 \rceil$. Furthermore, since $\sigma_{k+1} - \sigma_k \leq 2$, we also have $\sigma_j - \sigma_i \leq 2(j - i)$. Combining these, we get $L_i \leq j - i$.
\end{proof}

\subsection{Adaptation of the Infinite Collision Lemma}\label{sub3}
Xinxing Chen and Dayue Chen \cite{CC2011} establishes a sufficient condition for the infinite collision property on $Comb(\mathbb{Z},f)$. Let $f$ denote a function mapping $\mathbb{Z}$ into $\mathbb{R}^+$. $Comb(\mathbb{Z}, f)$ is defined as the set with vertex set $\mathbb{V}=\{(x,y):x,y\in \mathbb{Z}, -f(x)\leq y\leq f(x)\}$ and edge set $\mathbb{E}=\{\langle (x,y), (x,y')\rangle:|y-y'|=1\}\cup \{\langle (x,0), (x',0)\rangle:|x-x'|=1\}$. which is stated as follows.
\begin{lemma}\label{lem2}
Let $\tilde{f}(n) = 1 \vee \max_{-n \leq i \leq n} f(i)$. If $\sum_{n=1}^{\infty} \frac{1}{\tilde{f}(n)} = \infty$, then the comb graph $\mathrm{Comb}(\mathbb{Z}, f)$ has the infinite collision property.
\end{lemma}
 
\cite{CC2011} provides a rigorous proof for this. Here, we only discuss whether the above lemma holds for  $\mathrm{Comb}^*(\mathbb{Z}, f)$. The graph $\mathrm{Comb}^*(\mathbb{Z}, f)$ is defined to have vertex set $\mathbb{V'} = \{(x, y) : x, y \in \mathbb{Z}_{\geq 0}, 0 \leq y \leq g(x)\}$ and edge set $\mathbb{E'} = \{\langle (x, y), (x, y') \rangle : |y - y'| = 1\} \cup \{\langle (x, 0), (x', 0) \rangle : |x - x'| = 1\}$. In other words, $\mathrm{Comb}^*(\mathbb{Z}, f)$ is the part of the first quadrant of $\mathrm{Comb}(\mathbb{Z}, f)$.

Let $X = \{X_n, n \geq 0\}$ be a simple random walk on $\mathrm{Comb}(\mathbb{Z}, f)$. For $n \geq 0$, write $X_n = (x_n, y_n)$. Let $B_n = \{(x, y) \in \mathbb{V} : |x| \leq n\}$, and define the stopping time $\theta_n = \inf\{m \geq 0 : X_m \notin B_{n-1}\}$. Thus, if $X_0 \in \mathbb{V}_{n-1}$, then $\theta_n$ is the first hitting time of the vertices $\{(-n, 0), (n, 0)\}$ for $X$. 

Let $X' = \{X'_n, n \geq 0\}$ be another simple random walk on $\mathrm{Comb}(\mathbb{Z}, f)$, independent of $X$. Similarly, define $x'_n, y'_n, \theta'_n$ for $X'$.

Define another sequence of stopping times:
$$
\tau_{m+1} := \inf\{n > \tau_m : x_n = x'_n \text{ and } (x_n \neq x_{n-1} \text{ or } x'_n \neq x'_{n-1})\},
$$with $\tau_0 = 0$. Lemma 2.1 of \cite{CC2011}: For any $\epsilon > 0$, there exists $d \in \mathbb{N}$ such that
$$
P_{x, x'}(\tau_N \geq \theta_{dN} \wedge \theta'_{dN}) < \epsilon.
$$for all $N \in \mathbb{N}$ and all $x = (x_1, y_1), x' = (x'_1, y'_1) \in V$ with $x_1 + y_1 + x'_1 + y'_1$ even.

Let $\tilde{X}, \tilde{X'}$ are two independent simple random walks on $\mathrm{Comb}^*(\mathbb{Z}, f)$. Similarly, define $\tilde{x}_n, \tilde{y}_n, \tilde{\theta}_n$ for $\tilde{X}$, $\tilde{x'}_n, \tilde{y'}_n, \tilde{\theta'}_n$ for $\tilde{X'}$ and $\tilde{\tau_n}$. 

It suffices to discuss the variations of horizontal coordinates. Let $Z, Z', \tilde{Z}, \tilde{Z'}$ denote the horizontal coordinate evolution of $X, X', \tilde{X}, \tilde{X'}$. Since $P((n,0), (n+1,0))=P((n,0), (n-1,0)$ holds on $\mathrm{Comb}(\mathbb{Z}, f)$, the Strong Law of Large Numbers indicates that $Z, Z'$ are simple random walk on $\mathbb{Z}$. Similarly, $\tilde{Z}, \tilde{Z'}$ are simple random walk on $\mathbb{Z}_{\geq 0}$.

Let $\tilde{Z}(i)= |Z(i)|, \tilde{Z'}(i)= |Z'(i)|$, if $Z(j) = Z'(j)$ or $Z(j) = -Z'(j)$, then the corresponding coordinates in the non-negative comb are equal. Thus, $\tilde{\tau}_N \leq \tau_N$ holds trivially. Also, $\tilde{\theta}_{dN} = \theta_{dN}$ and $\tilde{\theta'}_{dN} = \theta'_{dN}$, so $\tilde{\theta}_{dN} \wedge \tilde{\theta'}_{dN} = \theta_{dN} \wedge \theta'_{dN}$. Therefore,
$$
\{\tilde{\tau}_N \geq \tilde{\theta}_{dN} \wedge \tilde{\theta'}_{dN}\} \subseteq \{\tau_N \geq \theta_{dN} \wedge \theta'_{dN}\},
$$

which implies

\begin{eqnarray}
P_{\tilde{x}, \tilde{x}'}(\tilde{\tau}_N \geq \tilde{\theta}_{dN} \wedge \tilde{\theta'}_{dN}) \leq P_{x, x'}(\tau_N \geq \theta_{dN} \wedge \theta'_{dN}) < \epsilon.
\end{eqnarray}

The rest of the proof is similar and is omitted.

\begin{lemma}\label{lem3}
For any random recursive tree $T$ generated by a Bernoulli sequence $\{b_n\}$, there exists a unique comb graph $\mathrm{Comb}^*(\mathbb{Z}, f)$ corresponding to it.
\end{lemma}

\begin{proof}
P\textbf{roof: }Construct a mapping $\Phi: T \to \mathbb{Z}_{\geq 0}^2$ as follows:

\begin{eqnarray}
\Phi(v) =
\begin{cases}
(i, 0), & \text{if } v = v_{\sigma_i} \in P_{v_0}; \\
(i, d(v, v_{\sigma_i})), & \text{if } v \in P'_{v_{\sigma_i}}.
\end{cases}
\end{eqnarray}

Define the function $f: \mathbb{Z}_{\geq 0} \to \mathbb{R}$ by $f(i) = L_i$. This defines a comb graph $\mathrm{Comb}^*(\mathbb{Z}, f)$. We can thus identify the image $\Phi(T)$ with $\mathrm{Comb}^*(\mathbb{Z}, f)$.

Note that the mapping $\Phi$ is not injective; that is, there can exist $T \neq T'$ such that $\Phi(T) = \Phi(T')$.
\end{proof}

\subsection{Proof of Theorem \ref{thm2}}
\begin{proof}\ 

We have coupled the random recursive tree $T$, generated by the sequence $\{b_n\}$, to $\Phi(T)$ via Lemma \ref{lem3}, and established the applicability of Lemma \ref{lem2} to $\Phi(T)$. And then, we only need to prove that $\sum_{n=1}^{\infty} \frac{1}{\tilde{f}(n)} = \infty $ on $\Phi(T)$ to show that the random recursive tree $T$ has the infinite collision property via Lemma \ref{lem2}. We used two methods to prove that the sum is infinite. 
 
\textbf{M1:} Define a new sequence of stopping times $\xi_{i+1} := \inf\{n > \xi_i : f(n) > f(\xi_i)\}$, with $\xi_0 = 0$. By Lemma \ref{lem1}, we have $\xi_{i+1} - \xi_i \geq f(\xi_i)$, and $f(j) = 0$ for $\xi_i + 1 \leq j \leq \xi_i + f(\xi_i) - 1$. Therefore,

\begin{eqnarray}
\sum_{n=\xi_i}^{\xi_{i+1}-1} \frac{1}{\tilde{f}(n)} \geq \sum_{n=\xi_i}^{\xi_i + f(\xi_i) - 1} \frac{1}{f(\xi_i)} = f(\xi_i) \cdot \frac{1}{f(\xi_i)} = 1.
\end{eqnarray}

If the set $\{\xi_n\}$ is infinite, then

\begin{eqnarray} 
  \sum_{n=1}^{\infty} \frac{1}{\tilde{f}(n)} = \sum_{i=0}^{\infty} \sum_{j=\xi_i}^{\xi_{i+1}-1} \frac{1}{\tilde{f}(j)} \geq \sum_{i=0}^{\infty} 1 = \infty.\label{001}
 \end{eqnarray}

If there exists $M > 0$ such that the set $\{\xi_n\}$ has size $M$, then for all $n > M$, we have $f(n) < f(M)$. Thus,

\begin{eqnarray}
  \sum_{n=1}^{\infty} \frac{1}{\tilde{f}(n)} = \sum_{n=1}^{M-1} \frac{1}{\tilde{f}(n)} + \sum_{n=M}^{\infty} \frac{1}{\tilde{f}(n)} \geq \sum_{n=1}^{M-1} \frac{1}{\tilde{f}(n)} + \sum_{n=M}^{\infty} \frac{1}{f(M)}.\label{002}
  \end{eqnarray}

The first part is a sum of finitely many terms, resulting in a finite value. The second sum is an infinite sum of a positive constant, which diverges to infinity. Hence, $\sum_{n=1}^{\infty} \frac{1}{\tilde{f}(n)} = \infty$.

In conclusion, from (\ref{001}) and (\ref{002}), $\sum_{n=1}^{\infty} \frac{1}{\tilde{f}(n)} = \infty$ holds.

\textbf{M2:} Let $S_N$ be the length of the longest run of 1s in a Bernoulli sequence of length $N$ with parameter $q$. Theorem 1.1 of \cite{MWW2015} gives a large deviation result for $S_N$: For each $x > 0$, \quad

\begin{eqnarray}  
 \lim_{N \to \infty} \frac{-1}{\ln N} \ln P\left[ \frac{S_N}{\ln N} \geq \eta(q) + x \right] = x \eta(q),
\end{eqnarray}
 
where $\eta(q) = (\ln \frac{1}{q})^{-1}$.

From this theorem, we obtain:

$$
P(S_N \geq (\eta(q) + x) \ln N) \sim N^{-x / \eta(q)}.
$$

Let $A_N = \{S_N \geq (\eta(q) + x) \ln N\}$. For $x > \eta(q)$, we have

\begin{eqnarray}
\sum_{N=1}^{\infty} P(A_N) \sim \sum_{N=1}^{\infty} N^{-x / \eta(q)} < \infty.
  \end{eqnarray}
 
By the Borel-Cantelli lemma, $P(A_N \text{ i.o.}) = 0$. That is, with probability 1, only finitely many of the events $A_N$ occur.

From Lemma \ref{lem1}, we know $\sigma_n \leq 2n$. For any $n \geq 1$, let $\theta_n$ be the length of the $0$s run containing $\sigma_n$. The above result implies that

\begin{eqnarray} 
P(\theta_n \geq (\eta(q) + x) \ln(n + \theta_n) \text{ occurs only finitely often}) = 1.
 \end{eqnarray}

There exists a constant $k_1$ such that for all $n$, $(\eta(q) + x) \ln(n + \theta_n) \leq k_1 n$. Therefore, if $b_{\sigma_n} = 1$, then $\tilde{f}(n) \leq S_{2n}$.
If $b_{\sigma_n} = 0$, then $\tilde{f}(n) \leq S_{2n} \vee S_{(k_1+2)n}$.

In summary, there exists $k_2 > 0$ such that $\tilde{f}(n) \leq S_{k_2 n}$. Consequently,

\begin{eqnarray} 
\sum_{i=1}^{\infty} \frac{1}{\tilde{f}(i)} \geq \sum_{i=1}^{\infty} \frac{1}{S_{k_2 i}} \geq \sum_{i=1}^{\infty} \frac{1}{(\eta(q)+x) \ln (k_2 i)} = \frac{1}{(\eta(q)+x)} \sum_{i=1}^{\infty} \frac{1}{\ln k_2i} = \infty \quad \text{a.s.}
\end{eqnarray}

This completes the proof of Theorem \ref{thm2}.
\end{proof}


\noindent School of Mathematical Sciences, Capital Normal University, Beijing, 100048, People R China.\\
\indent Email:15213277742@163.com, wangf@cnu.edu.cn, wuxy@cnu.edu.cn

\end{document}